\documentclass[a4paper,12pt,tbtags]{amsart}

\usepackage[T1]{fontenc}
\usepackage[utf8]{inputenc}
\usepackage{amsfonts,amssymb}
\usepackage{mathrsfs}
\usepackage{graphicx}
\usepackage{hyperref}

\theoremstyle{plain}
\newtheorem{theorem}{Theorem}
\newtheorem{proposition}[theorem]{Proposition}
\newtheorem{lemma}[theorem]{Lemma}

\theoremstyle{definition}
\newtheorem{definition}{Definition}

\numberwithin{equation}{section}

\begin{document}

\title{Generically Flexible Varieties with an Invariant Divisor}
\author{Alina Nikitina}
\address{Faculty of Computer Science, HSE University, Pokrovsky Boulevard 11, Moscow, 109028 Russia}
\email{avnikitina@hse.ru}


\begin{abstract}
In this paper, we construct a series of smooth affine varieties of dimension $n \geqslant 3$ that are generically flexible but not flexible. Moreover, these varieties admit a nontrivial divisor invariant under the subgroup of special automorphisms.
\end{abstract}

\maketitle

\section{Introduction}\label{sec1}

Let $\mathbb{K}$ be an algebraically closed field of characteristic $0$, and let $\mathbb {G}_a$ denote the additive group over $\mathbb{K}$. For an irreducible affine variety $X$, denote by $\operatorname{Aut}(X)$ the group of all its regular automorphisms, and by $\operatorname{SAut}(X)$ the subgroup generated by all $\mathbb {G}_a$-subgroups. 

Denote by $X^{\mathrm{reg}}$ the set of smooth points of $X$. A point $x \in X^{\mathrm{reg}}$ is called \textit{flexible} if the tangent space $T_xX$ is spanned by the tangent vectors to the orbits of $\mathbb {G}_a$-actions. The variety $X$ is called \textit{flexible} if all its smooth points are flexible. Flexibility of affine varieties is related to the action of the group $\operatorname{SAut}(X)$. For an irreducible affine variety $X$ of dimension at least $2$, flexibility is equivalent to transitivity of the action of $\operatorname{SAut}(X)$ on $X^{\mathrm{reg}}$, and, moreover, it is equivalent to infinite transitivity of this action \cite{5A}.

The flexibility property has been actively studied for various classes of affine varieties; see, for example, \cite{AZK}--\cite{ShafarevichFlexibility}. These results either establish flexibility for specific varieties or, conversely, give examples that are not flexible.

Consider a weaker property. We say that $X$ is \textit{generically flexible} if there exists a nonempty open subset consisting of flexible points. Equivalently, this condition is the same as the existence of at least one flexible point on $X$ \cite{5A}. Unlike flexibility, generic flexibility allows for nontrivial $\operatorname{SAut}(X)$-invariant subsets in the set of smooth points.

In \cite{Kovalenko}, a series of smooth Gizatullin surfaces $X$ is constructed that are generically flexible but not flexible: each of them is the union of an open $\operatorname{SAut}(X)$-orbit and a finite set of $\operatorname{SAut}(X)$-fixed points. Thus, in these examples the complement of the open orbit is finite.

In \cite{GDivisor}, for each dimension $m \geqslant 4$, a normal affine variety $X$ is constructed on which the group $\operatorname{SAut}(X)$ has an open orbit $\mathcal O$, and the complement $X \setminus \mathcal O$ has codimension one. Like the previous examples, these varieties are generically flexible but not flexible. In this case the complement of the open orbit is not finite, but forms a divisor. At the same time, the varieties constructed in \cite{GDivisor} are not smooth. It is therefore natural to consider the analogous problem in the smooth case.

In this paper we construct a series of smooth affine varieties that are generically flexible but not flexible and, moreover, admit a nontrivial $\operatorname{SAut}(X)$-invariant divisor. The main result of the paper is stated in Theorem \ref{mainthm}.

\section{Preliminaries}
\label{subsec2}

In this section we introduce the main notions used in the paper: locally nilpotent derivations and the Makar-Limanov invariant, factoriality, Weil divisors, and, finally, Cox rings, which in Section~\ref{subsec3} serve as a bridge between the variety $X$ and its torus quotient $Y=X/\!\!/\mathbb G_m$: with their help, $\mathbb G_m$-invariant locally nilpotent derivations on $X$ are transferred to $Y$, which lets us establish generic flexibility of $Y$ and construct an invariant divisor on it.

Let $\mathbb{K}$ be an algebraically closed field of characteristic $0$, let $A$ be a finitely generated algebra over $\mathbb{K}$, and let $X=\mathrm{Spec}\,A$.

\begin{definition}
A derivation of the algebra $A$ is a linear map $D\colon A\to A$ satisfying the Leibniz rule for all $a,b\in A$:
$$D(ab)=aD(b)+bD(a).$$
\end{definition}

\begin{definition}
A derivation $D\colon A\to A$ is called locally nilpotent if for every $a\in A$ there exists $n\in\mathbb{N}$ such that $D^n(a)=0$.
\end{definition}

Denote the set of all such derivations by $\operatorname{LND}(A)$.

\begin{definition}
The kernel of a derivation $D$ is
$$\operatorname{Ker}D=\{f\in A\mid D(f)=0\}.$$
\end{definition}

Let $G$ be an abelian group and let $A=\bigoplus_{h\in G}A_h$ be a $G$-grading of the algebra $A$.

\begin{definition}
A derivation $D$ is called homogeneous of degree $g\in G$ if for every homogeneous element $a\in A_h$ we have
$$D(a)\in A_{h+g}.$$
\end{definition}

\begin{definition}
The Makar-Limanov invariant of the algebra $A$ is the subalgebra
$$\operatorname{ML}(A)=\!\!\!\!\!\!\bigcap_{D\in\operatorname{LND}(A)}\!\!\!\!\!\!\operatorname{Ker}D.$$
\end{definition}
For a variety $X$ we set $\operatorname{ML}(X)=\operatorname{ML}(\mathbb{K}[X])$.

If an algebraic torus $T$ acts on an affine variety $X$, then the categorical quotient by this action is
$$\pi\colon X\to X/\!\!/T:=\mathrm{Spec}\,A^T,$$
where $A^T$ is the subalgebra of $T$-invariants in $A=\mathbb K[X]$. If the action of the torus $T$ corresponds to the grading of $A$ by the character group $\mathfrak{X}(T)$, then $A^T=A_0$ is the zero graded component. In Section~\ref{subsec3} the notation $\pi$ is used precisely for this quotient map when $T=\mathbb G_m$.

\subsection{Cox rings}
\label{subsec2.1}

Let $X$ be an irreducible normal algebraic variety over an algebraically closed field $\mathbb{K}$, with trivial group of invertible regular functions and finitely generated divisor class group $\operatorname{Cl}(X)$. Denote by $\operatorname{WDiv}(X)$ the group of Weil divisors on $X$.

Fix a subgroup $K\subset \operatorname{WDiv}(X)$ that maps onto $\operatorname{Cl}(X)$ surjectively. More details can be found in the book \cite{CoxB}. Then the algebra of divisorial sections is defined as
$$T_K(X)=\bigoplus_{\delta\in K}\mathcal{O}(X,\delta),\qquad 
\mathcal{O}(X,\delta)=\{f\in \mathbb{K}(X)\mid \operatorname{div}(f)+\delta\geqslant 0\}.$$

Let $K^0$ denote the kernel of the map $K\to \operatorname{Cl}(X)$, and choose a system of rational functions $\mathcal{F}=\{F_\delta\in\mathbb{K}(X)^\times \mid \delta\in K^0\}$, satisfying $\operatorname{div}(F_\delta)=\delta$ and $F_{\delta+\delta'}=F_\delta F_{\delta'}$.

\begin{definition}
The Cox ring of the variety $X$ is the quotient ring
$$R(X)=T_K(X)\big/ I(K,\mathcal{F}),$$
where $I(K,\mathcal{F})$ is the ideal generated by elements of the form
$$f - F_{\delta_1-\delta_2}f,\qquad f\in \mathcal{O}(X,\delta_1),\; \delta_1,\delta_2\in K,\;\delta_1-\delta_2\in K^0.$$
\end{definition}

The ring $R(X)$ does not depend on the particular choice of $K$ and $\mathcal F$, up to isomorphism \cite[Proposition 3.2]{ArzFact}, and it is graded by the group $\operatorname{Cl}(X)$.

Let $M$ be a finitely generated abelian group. Consider an algebra $$R={\bigoplus\limits_{u\in M}}R_u,$$ graded by the group $M$.

\begin{definition}
We say that a nonzero homogeneous element $a\in R\setminus R^\times$ is $M$-irreducible if for any homogeneous $b,c$ with $a=bc$, either $b$ or $c$ is invertible.
\end{definition}

\begin{definition}
An algebra $R$ graded by the group $M$ is called factorially graded if it has no nonconstant homogeneous invertible elements, and every nonconstant homogeneous element decomposes as a product of $M$-irreducible elements, and this decomposition is unique up to permutation and association of the factors.
\end{definition}

Consider a quasitorus $N$ with a fixed isomorphism $\mathfrak X(N)\cong\operatorname{Cl}(X)$. It acts on $R(X)$ by homogeneous automorphisms, and the subalgebra of invariants coincides with the zero component: $R(X)^N=R(X)_0=\mathbb K[X]$.

Assume further that $R(X)$ is finitely generated. Consider $\overline X:=\operatorname{Spec}R(X)$ with the induced regular action of $N$. Since $X\cong\operatorname{Spec}R(X)^N$, the variety $X$ is realized as the categorical quotient $\overline X/\!\!/N$; the quotient morphism
$$q\colon\overline X\to X$$
is called the \textit{Cox realization} of the variety $X$. It is known that $\overline X$ is irreducible and normal \cite[Prop.~2.2]{AG}. \\

\noindent{\sc Proposition}  (2.3, \cite{AG}).\label{prop2.3}
Let $q\colon\overline X\to X$ be the Cox realization of an affine variety $X$. Let a quasitorus $Q$ act on an irreducible affine variety $Z$, so that $Z$ is factorially graded by this action. Let $\pi\colon Z\to Z/\!\!/Q$ be the quotient morphism, and fix an isomorphism $\varphi\colon X\to Z/\!\!/Q$. Then the following conditions are equivalent:
\begin{enumerate}
    
    \item[(i)] there exists an open $Q$-invariant subset $U\subseteq Z$ such that $\operatorname{codim}_Z(Z\setminus U)\geqslant 2$, the action of $Q$ on $U$ is free, and every fiber of $\pi$ that meets $U$ meets it in exactly one $Q$-orbit;

    \item[(ii)] there exist an isomorphism $\mu\colon Q\to N$ and an isomorphism $\xi\colon Z\to\overline X$ such that $\xi(gz)=\mu(g)\xi(z)$ for all $g\in Q,\,z\in Z$, and such that the following diagram commutes:   
    $$\begin{array}{ccc}
Z                   & \xrightarrow{\quad\xi\quad}      & \overline{X} \\
{\pi}\big\downarrow &                                  & \big\downarrow{q} \\
Z/\!\!/Q            & \xrightarrow{\quad\varphi\quad } & X
\end{array}$$ 

\end{enumerate}

More on Cox rings can be found here \cite{CoxB}.

\section{Main result}
\label{subsec3}

Fix a natural number $n$ and consider the $(n+3)$-dimensional affine variety $$X=\{x^2y+z^3w^2u_1^2\cdots u_n^2=1\}\subseteq\mathbb{A}^{n+4}.$$

Define an action of the torus $\mathbb{G}_m$ on $X$, induced by the grading:
\begin{gather*}
    \deg x=1, \quad \deg y=-2, \quad \deg z=2, \\
    \deg w=-2, \quad \deg u_1=-1, \quad \deg u_i=0 \quad (i=2, \dots, n).
\end{gather*} 

The main result of the paper is the following theorem.

\begin{theorem}\label{mainthm}
For every $n\geqslant 1$, the affine variety $Y=X/\!\!/\mathbb{G}_m$ is smooth, generically flexible, but not flexible. Moreover, $Y$ admits a nontrivial $\operatorname{SAut}(Y)$-invariant divisor.
\end{theorem}

To prove this theorem we need a number of auxiliary statements.

\begin{proposition}\label{MX}
 $\operatorname{ML}(X)=\mathbb{K}[x].$
\end{proposition}

\begin{proof}
The equation of $X$ is trinomial, and in the monomial $z^3w^2u_1^2\cdots u_n^2$ all exponents are at least two, while $y$ enters $x^2y$ linearly with the degree of $x$ greater than one. Applying Proposition 2 of \cite{GTrinomial} gives the required equality. This proves Proposition \ref{MX}.
\end{proof}

Note that the action of $\mathbb {G}_m$ on $X$ is free. Let a point $p=(x,\dots,u_n)\in X$ be fixed by an element $t\in\mathbb {G}_m$. If $x\neq0$, then $tx=x$, so $t=1$. If $x=0$, then $z^3w^2u_1^2\cdots u_n^2=1$, so $z,u_1\neq0$. From $t^2z=z$ and $t^{-1}u_1=u_1$ we get $t=1$. Hence the action is free.

\begin{lemma}\label{Fact}
The variety $X$ is factorial.    
\end{lemma}
    
\begin{proof}
The coordinate ring of $X$ is
$$R=\mathbb{K}[x,y,z,w,u_1,\ldots,u_n]\big/(x^2y+z^3w^2u_1^2\cdots u_n^2-1),$$
and it is a trinomial algebra of type $I$ in the sense of Construction 1.1 of \cite{HausenWrobel}.

For the corresponding exponents $l_0=(2,1),\; l_1=(3,2,2,\ldots,2)$ we have $\gcd(l_{ij})=1$. By the criterion of \cite{HausenWrobel}, $R$ is factorial, hence $X$ is factorial. 
\end{proof}

\begin{proposition}\label{CoxR}
The quotient morphism $X\to Y$ is the Cox realization of the variety $Y$.
\end{proposition}

\begin{proof}
Since $X$ is factorial, it is normal. The weight group of the grading on $\mathbb K[X]$ is $\mathbb Z$, so $\mathbb K[X]$ is factorially graded with respect to the action of $\mathbb G_m$.

The subalgebra of invariants $\mathbb K[Y]=\mathbb K[X]^{\mathbb G_m}$ of the integrally closed ring $\mathbb K[X]$ is also integrally closed, since an element $\xi\in\operatorname{Frac}(\mathbb K[Y])$ that is integral over $\mathbb K[Y]$ is also integral over $\mathbb K[X]$, so $\xi\in\mathbb K[X]\cap\operatorname{Frac}(\mathbb K[Y])=\mathbb K[Y]$. Hence $Y$ is normal.

Let us check condition (i) of Proposition 2.3 from \cite{AG}. Set $U=X$. Then $U$ is open, $\mathbb G_m$-invariant, and $\operatorname{codim}_X(X\setminus U)\geqslant 2$. 
Since the action of $\mathbb G_m$ on $X$ is free, for every point $p\in X$
$$\dim(\mathbb G_m p)=\dim\mathbb G_m=1.$$
All orbits have the same dimension and are closed. Hence every fiber of the morphism $\pi\colon X\to Y$ consists of a single $\mathbb G_m$-orbit, and condition (i) holds.

By Proposition~2.3 from \cite{AG}, the quotient morphism $X\to Y$ is the Cox realization of $Y$. This proves Proposition~\ref{CoxR}.

\end{proof}

Consider the following invariants of weight zero in $\mathbb{K}[X]$:
$$A=x^2w,\quad B=zw,\quad C=yz,\quad D=xu_1,\quad E=zu_1^2,\quad U_i=u_i\quad(i=2,\dots,n).$$

Denote $P = U_2^2\cdots U_n^2$ and set 
$$f_1=AE-BD^2,\qquad f_2=B-AC-EB^3P,\qquad f_3=E-CD^2-E^2B^2P,$$

satisfying the identities
$$Bf_3-Cf_1-Ef_2=0,\qquad Af_3-D^2f_2-(1-EB^2P)f_1=0. \eqno{(*)}$$

The Jacobian matrix $J$ of the system $\{f_1,f_2,f_3\}$ with respect to the variables $A,B,\dots,U_n$ is:

$$J=\begin{pmatrix}
E & -2BD & 0 & -2BD & A & \dots & 0 \\
-C & 1-3EB^2P & -A & 0 & -B^3P & \dots & -2EB^3P/U_n \\
0 & -2E^2BP & -D^2 & -2CD & 1-2EB^2P & \dots & -2E^2B^2P/U_n
\end{pmatrix}.$$

\begin{lemma}\label{Rank}
The rank of the matrix $J$ equals $2$ at every point of the set $V(f_1,f_2,f_3)$.
\end{lemma}

\begin{proof}
The last $n-1$ columns of $J$ are proportional, so the rank of $J$ equals the rank
of the matrix of columns $(A,B,C,D,E,U_2)$.

Differentiating $(*)$, we get the following relations among the rows of the matrix:
$$B\nabla f_3=C\nabla f_1+E\nabla f_2,\qquad\qquad A\nabla f_3=D^2\nabla f_2+(1-EB^2P)\nabla f_1.$$

Suppose $A\neq0$. The second equality expresses $\nabla f_3$ in terms of $\nabla f_1,\nabla f_2$, so $\operatorname{rk}J\leqslant2$. In the column $C$ the coordinates are $0$ and $-A$, and in the column $E$ they are $A$ and $0$, so the rows $\nabla f_1,\nabla f_2$ are independent, and $\operatorname{rk}J=2$.

Now let $A=0$. If $B\neq0$, the first equality expresses $\nabla f_3$ in terms of $\nabla f_1,\nabla f_2$, so $\operatorname{rk}J\leqslant2$. When $B\neq0$ we have $D=0$ and $E,P\neq0$. If $\operatorname{rk}J\leqslant1$, then the columns $A$ and $E$ give $EB^3P=0$, which is impossible. Hence $\operatorname{rk}J=2$.

So $B=0$. From $f_3=0$ we get $E=CD^2$, so $\nabla f_1=-D^2\nabla f_2$ and $\operatorname{rk}J\leqslant2$. If $\operatorname{rk}J\leqslant1$, then $\nabla f_3=\lambda\nabla f_2$ for some $\lambda$. In the column $B$ the coordinate of $\nabla f_2$ is $1$, and of $\nabla f_3$ is $0$, so $\lambda=0$ and $\nabla f_3=0$. But in the column $E$ the coordinate of $\nabla f_3$ is $1\neq0$, a contradiction. Hence $\operatorname{rk}J=2$. This proves Lemma \ref{Rank}.

\end{proof}

\begin{lemma}\label{Invts}
There is an isomorphism of $\mathbb K$-algebras
$$\mathbb{K}[X]^{\mathbb{G}_m}\cong \mathbb{K}[A,B,C,D,E,U_2,\dots,U_n] \big/ (f_1,f_2,f_3).$$
\end{lemma}

\begin{proof}
Consider the homomorphism sending the variables to the invariants:
$$\psi\colon\mathbb K[A,B,C,D,E,U_2,\dots,U_n]\longrightarrow\mathbb K[X]^{\mathbb G_m}.$$

Let us show that every monomial 
$$M=x^a y^b z^c w^d u_1^{m_1}u_2^{m_2}\cdots u_n^{m_n}$$
of weight zero can be expressed in terms of $A,B,C,D,E,U_2,\dots,U_n$. 
Since $u_i$ ($i\geqslant2$) are already invariant, it suffices to consider the monomial $M'=x^ay^bz^cw^du_1^{m_1}$ of weight $$a-2b+2c-2d-m_1=0.$$

We use induction on the degree $a+b+c+d+m_1$.

If $m_1>0$ and $a>0$, then $M'$ is divisible by $D=xu_1$, and the quotient has smaller degree.

If $m_1>0$ and $a=0$, then $m_1=2(c-b-d)$. If $c=0$, then $m_1\leqslant0$. Contradiction. So $c>0$, and $M'$ is divisible by $E=zu_1^2$.

If $m_1=0$, the weight condition becomes $a+2c=2b+2d$. Then:

\begin{enumerate}
\item if $d>0$ and $c>0$, then $M'$ is divisible by $B$;
\item if $d>0$ and $c=0$, then $a=2(b+d)\geqslant2$, and $M'$ is divisible by $A$;
\item if $d=0$ and $c>0$, then $b\neq0$, and $M'$ is divisible by $C$;
\item if $d=0$ and $c=0$, then $a=2b$, and $M'=(x^2y)^b$.
\end{enumerate}

In the first three cases the quotient has smaller degree, and the claim holds by the induction hypothesis. In the fourth case $x^2y=1-B^2EP$ is expressed in terms of $B,E,U_2,\dots,U_n$. \medskip

Denote $$S=\mathbb K[A,B,C,D,E,U_2,\dots,U_n]/(f_1,f_2,f_3).$$ Since $f_1,f_2,f_3\in\ker\psi$, the map $\psi$ induces a surjection $\overline\psi\colon S\to\mathbb K[X]^{\mathbb G_m}$.

Consider the localization $S_A$. From $f_1=0$ we get $E=BD^2/A$. From $f_2=0$ we get $C=(AB-B^4D^2P)/A^2$, and $f_3$ then vanishes identically.
Hence
$$S_A\cong\mathbb K[A^{\pm1},B,D,U_2,\dots,U_n],$$
and $S_A$ is an integral domain of dimension $n+2$.

Let us show that $V(f_1,f_2,f_3)$ is irreducible. By Lemma \ref{Rank}, this variety is smooth everywhere and of dimension $n+2$; in particular, exactly one irreducible component passes through each of its points.

Denote $\mathcal D_A:=\{A=0\}$. Let us show that $\dim(\mathcal D_A\cap V(f_1,f_2,f_3))\leqslant n+1$, so that $\operatorname{Spec}S_A=\{A\neq0\}\cap V(f_1,f_2,f_3)$ is dense in $V(f_1,f_2,f_3)$.

At $A=0$, the equations $f_1=0,f_2=0$ give $BD^2=0$ and $B(1-EB^2P)=0$.
Then:
\begin{enumerate}
\item if $B=0$, then $f_3=0$ gives $E=CD^2$;
\item if $B\neq0$, then $D=0$, $EB^2P=1$.
\end{enumerate}
In both cases $\dim(\mathcal D_A\cap V(f_1,f_2,f_3))\leqslant n+1$.

Take a point $q\in \mathcal D_A$. If $q\notin\overline{\operatorname{Spec}S_A}$, then
some neighborhood of $q$ lies in $\mathcal D_A$ and, by smoothness, has dimension $n+2$.
This contradicts the estimate above. Hence $q\in\overline{\operatorname{Spec}S_A}$. Since this
holds for every point of $\mathcal D_A$, the whole set $V(f_1,f_2,f_3)$ coincides with
the closure of $\operatorname{Spec}S_A$, that is, it is irreducible. Therefore $S$ is an integral
domain of dimension $n+2$.
Thus $\overline\psi$ is a surjective homomorphism between integral domains of dimension $n+2$, hence injective, and so it is an isomorphism. This proves Lemma \ref{Invts}.

\end{proof}

Denote $\mathcal{A}=\mathbb{K}[X]$. Consider $n+2$ homogeneous locally nilpotent derivations of degree $0$, corresponding to the variables $z, w, u_1, \dots, u_n$.

Their action is given by the following formulas:
$$\partial_z(x)=\partial_z(w)=\partial_z(u_i)=0,\qquad \partial_z(z)=x^2,\qquad \partial_z(y)=-3z^2w^2u_1^2P,$$
$$\partial_w(x)=\partial_w(z)=\partial_w(u_i)=0,\quad \partial_w(w)=x^2u_1^4,\quad \partial_w(y)=-2z^3wu_1^6P,$$
$$\partial_1(x)=\partial_1(z)=\partial_1(w)=\partial_1(u_j)=0\, (j>1),\quad \partial_1(u_1)=x^3w^2,\quad \partial_1(y)=-2xz^3w^4u_1P,$$
and for all $i=2,\dots,n$:
$$\partial_i(x)=\partial_i(z)=\partial_i(w)=\partial_i(u_j)=0\ (j\ne i),\quad \partial_i(u_i)=x^2w,\quad \partial_i(y)=-2z^3w^3u_1^2u_i^{-1}P.$$

\begin{lemma}\label{GFlex}
The variety $Y$ is generically flexible.
\end{lemma}

\begin{proof}
The derivations $\partial_z,\partial_w,\partial_1,\dots,\partial_n$ preserve $\mathcal A_0=\mathbb K[Y]$ and induce $\overline\partial_z,\overline\partial_w,
\overline\partial_1,\dots,\overline\partial_n\in\operatorname{LND}(\mathcal A_0)$. 

Consider the matrix $M$ of values of these derivations on $(A,B,C,\dots,U_n)$:
$$M=\begin{pmatrix}
0 & A & 1-4B^2EP & 0 & D^2 & 0 & \dots & 0\\
D^4 & D^2E & -2BE^3P & 0 & 0 & 0 & \dots & 0\\
0 & 0 & -2B^4DP & A^2 & 2ABD & 0 & \dots & 0\\
0 & 0 & -2B^3EU_2^{-1}P & 0 & 0 & A & \dots & 0\\
\vdots & \vdots & \vdots & \vdots & \vdots & \vdots & \ddots & \vdots\\
0 & 0 & -2B^3EU_n^{-1}P & 0 & 0 & 0 & \dots & A
\end{pmatrix}.$$

Denote by $M_0$ the minor corresponding to the columns $(A,B,D,U_2,\dots,U_n)$. Expanding along the last $n-1$ rows, we get 
$$M_0=A^{n-1}\det\begin{pmatrix}0&A&D^2\\D^4&D^2E&0\\0&0&A^2\end{pmatrix}=-A^{n+2}D^4.$$

Take the point $p=(1,1,0,1,1,\dots,1)\in X$, at which $A(p)=D(p)=1\neq0$. Then at the point $q=\pi(p)\in Y$ we have $M_0(q)=-1\neq0$, that is, $\operatorname{rk}M(q)=n+2=\dim T_qY$. Hence $\overline\partial_z(q),\dots,\overline\partial_n(q)$ form a basis of the space $T_qY$, and $q$ is a flexible point. By \cite{5A}, $Y$ is generically flexible. This proves Lemma \ref{GFlex}.

\end{proof}

Define the divisor $\mathcal{D}$ as the image of the invariant subset $\{x=0\}\subset X$
under the map $\pi\colon X\to Y$.

\begin{proposition}\label{Div}
The divisor $\mathcal D$ is $\operatorname{SAut}(Y)$-invariant.
\end{proposition}

\begin{proof}
First we show that $\mathcal D$ is indeed a divisor on $Y$. The set $\{x=0\}$ has codimension 1 in $X$ and is $\mathbb G_m$-invariant. Since the action of $\mathbb G_m$ on $X$ is free, all its orbits have dimension $1$, so
$$\dim\mathcal D=\dim\{x=0\}-1=\dim Y-1.$$

Now we show that $\mathcal D$ is invariant. Since $\pi\colon X\to Y$ is a Cox realization of the variety $Y$ (Proposition \ref{CoxR}), by \cite[Theorem 4.2.3.2(i)]{CoxB}, for every $\delta_0\in\operatorname{LND}(\mathcal A_0)$ there exists $\delta\in\operatorname{LND}(\mathcal A)_0$ such that
$$\pi\circ\exp(t\delta)=\exp(t\delta_0)\circ \pi \qquad\text{for all }t.$$
It suffices to check invariance of $\mathcal D$ under each subgroup $\exp(t\delta_0)$. By Proposition \ref{MX}, the set $\{x=0\}$ is invariant under $\exp(t\delta)$. Therefore
$$\exp(t\delta_0)(\mathcal D) =\exp(t\delta_0)(\pi(\{x=0\})) =\pi(\exp(t\delta)(\{x=0\})) =\pi(\{x=0\})=\mathcal D.$$
Hence $\mathcal D$ is invariant under the whole group $\operatorname{SAut}(Y)$.
\end{proof}

We can now prove Theorem \ref{mainthm}. 

\begin{proof}
Smoothness of $Y$ follows from Lemma \ref{Rank} and Lemma \ref{Invts}. Generic flexibility of $Y$ is proved in Lemma \ref{GFlex}. By Proposition \ref{Div}, the divisor $\mathcal D\subset Y$ is invariant under $\operatorname{SAut}(Y)$, so the action of $\operatorname{SAut}(Y)$ on $Y$ is not transitive. By \cite{5A}, this is equivalent to $Y$ not being flexible. This proves Theorem \ref{mainthm}.
\end{proof}

\section*{Acknowledgements}
The author is deeply grateful to Sergei Alexandrovich Gaifullin, under whose supervision this work was carried out, for suggesting the problem and for many valuable discussions.

\end{document}